\documentclass[12pt]{article}
\usepackage[top=1in, bottom=1in, left=1in, right=1in]{geometry}
\usepackage{amssymb}
\usepackage{float,epsfig, floatflt,here}
\usepackage{xcolor,hyperref}
\usepackage{graphicx}
\usepackage{amsfonts,calc}
\usepackage{amsthm}
\usepackage{tikz}
\usepackage{pgf}
\usepackage{pgfplots}

\usepackage{amssymb}

\usepackage{pgf,tikz,amsmath}
\usepackage{blkarray}
\usepackage{caption}
\usepackage{enumerate}
\usepackage{xspace}
\usepackage{setspace}
\usepackage{bm}
\usepackage{array}
\usepackage{longtable}
\usepackage{placeins}

\usepackage{float}
\usepackage{graphicx}
\usetikzlibrary{arrows}
\usetikzlibrary{decorations.markings}
\tikzset{->-/.style={decoration={
markings,
mark=at position #1 with {\arrow{>}}},postaction={decorate}}}

\usepackage{graphics}
\usepackage{graphicx}

\def\bpsp{\begin{pspicture}}
	\def\epsp{\end{pspicture}}
\newcommand\arrow[2]{\rotatebox{#1}{\scalebox{2}{\psline[linecolor=#2]{->}(0,0)(.03,0)}}}

\newtheoremstyle{plain}
{\topsep}
{\topsep}
{\itshape}
{}
{\bfseries}
{.}
{ }
{\thmname{#1}\thmnumber{ #2}\thmnote{ (#3)}}

\theoremstyle{plain}
\newtheorem{theorem}{Theorem}[section]
\newtheorem{lemma}[theorem]{Lemma}
\newtheorem{corollary}[theorem]{Corollary}

\theoremstyle{definition}
\newtheorem{definition}[theorem]{Definition}
\newtheorem{remark}[theorem]{Remark}
\newtheorem{Conjecture}[theorem]{Conjecture}

\title{A maximum matching-based refinement of Brouwer’s conjecture }
\author{ Tahir Shamsher\footnote{Corresponding author. Department of Mathematics, IIT Bhubaneswar, Bhubaneswar, 752050, India (tahir.maths.uok@gmail.com).} 
 }
\date{}
\begin{document}
\maketitle
\begin{abstract}
Let $G$ be a simple graph on $n$ vertices and $e(G)$ edges. Let $\mu_1\geq \cdots \geq \mu_{n-1}\geq \mu_n=0$ be the Laplacian eigenvalues of $G$. For $k=1, \ldots, n$, let $S_k(G)=\sum_{i=1}^{k}\mu_i$. Brouwer’s conjecture asserts that for any $k\in\{1,\ldots, n\}$, $S_k(G)\leq e(G)+\binom{k+1}{2}$. In [Bounding the sum of the largest Laplacian eigenvalues of graphs, {\em Discrete Appl. Math.}, 170:95--103, (2014)], Rocha and Trevisan showed that the conjecture holds true for $1 \leq k \leq \lfloor g/5 \rfloor$, where $g$ denotes the girth of $G$. This bound on $k$ was later improved by Chen in [Improved results on Brouwer’s conjecture for sum of the Laplacian eigenvalues of a graph, {\em Linear Algebra Appl.}, 557:327--338, (2018)], who established that the conjecture holds for $1 \leq k \leq \lfloor g/4 \rfloor$. In this article, we further strengthen these results by proving that the Brouwer’s conjecture holds for $1\leq k\leq\left\lfloor \frac{m(G)}{2}\right\rfloor,$ where $m(G)$ denotes the matching number of $G$. Since $m(G)\geq \lfloor g/2\rfloor$, the case constitutes a genuine improvement over the aforementioned results. As an application, we show that if $G$ is a graph of order $n$ and with a perfect matching, then the Brouwer’s conjecture holds for $1\leq k\leq \left\lfloor \frac{n}{4}\right\rfloor$. Finally, we provide a new perspective on verifying Brouwer’s conjecture by proving that $G$ satisfies the Brouwer's conjecture if and only if for a fixed positive integer $h$, $\mathcal{S}_h(\overline{G})\leq e(\overline{G})+\binom{h+1}{2}$ holds whenever $\mathcal{S}_h(G)\leq e(G)+\binom{h+1}{2}$, where $\overline{G}$ denotes the complement of $G$.
\end{abstract}
	

	
\section{Introduction}\label{sec1}
Let $G=(V,E)$ be a graph with the vertex set $V=[n]:=\{1,\ldots,n\}$ and edge set $E$. Let $e(G)=|E|$ denote the number of edges in $G$, and let $\overline{G}$ denote the complement of $G$. For each vertex $i\in V$, let $N_G(i)$ denote the set of its neighbours and $d_i(G)=|N_G(i)|$ denote its degree in $G$. Throughout this paper, we assume that the vertices are labelled so that the degree sequence is nonincreasing, that is, $d_1(G)\geq d_2(G)\geq\cdots\geq d_n(G)$. If $G$ is connected, then any degree sequence $d(G)=(d_1(G),\ldots,d_n(G))$ arising this way is an $n$-partition of $2e(G)$. For $i\in [n]$, the {\em conjugate} of $d_i$ is defined as $d_i^*(G)=|\{j\in V:d_j(G)\geq i\}|$. Note that $d_n^*(G)=0$. 

 

If $H$ is a subgraph of the graph $G$, $G-H$ refers to the graph obtained by removing all the edges in $H$ from $G$. A {\em matching} in a graph $G$ is a collection of vertex-disjoint edges. The maximum size of a matching in $G$ is known as the {\em matching number} of $G$, denoted by $m(G)$. A matching is called a {\em perfect matching} if it covers every vertex of $G$. In that case $m(G)=\frac{n}{2}$. A {\em vertex cover} in $G$ is a subset of $V$ that includes at least one endpoint of every edge in $G$. The {\em covering number} of $G$, denoted by $\tau(G)$, is defined as the minimum number of vertices in the vertex cover of $G$. A {\em clique} in $G$ is a complete subgraph. The order of the maximum clique in $G$ is called the {\em clique number} of $G$. The {\em girth} of $G$, denoted by $g$, is the length of the shortest cycle in $G$. If $G$ contains no cycles (that is, it is acyclic), then the girth is defined to be infinite.
 
By $K_{n}$, we denote the complete graph on $n$ vertices. A {\em threshold graph} (also known as a nested split graph) is a graph without induced subgraphs isomorphic to the cycle on four vertices, or the path on four vertices, or two disjoint copies of $K_2$. Threshold graphs admit several equivalent definitions; in particular, they can be recursively defined using a binary code. In a recursive process, we start with an isolated vertex, and at each step a vertex is added either as a new isolated vertex or as a dominating vertex, that is, adjacent to all the previous vertices. In this way, a threshold graph with $n$ vertices is represented by a binary sequence $\mathbf{b}=b_1 \cdots b_n$, where $b_{i}$ is $0$ if the vertex $v_{i}$ is added as an isolated vertex, and it is $1$ if $v_{i}$ is added as a dominating vertex. If we denote by powers the repetitions in the binary sequence, then we can write $\mathbf{b}=0^{m_1}1^{n_1}\cdots 0^{m_r}1^{n_r}$. Let $T(\mathbf{b})=T(b_1\cdots b_n)$ denote the threshold graph obtained with the binary sequence $\mathbf{b}=b_1\cdots b_n$.

A {\em pineapple graph} $P_{n,f}$ is a graph on $n$ vertices obtained from a complete graph $K_{f+1}$ on $f+1$ vertices by adding $n-f-1$ pendant vertices to one vertex of $K_{f+1}$ (see Figure \ref{fig1}). It is easy to see that $P_{n,f}=T(0^11^{f-1}0^{n-f-1}1)$ is a threshold graph of trace $f$.
 
\begin{figure*}[!htb] 
\centering
\begin{tikzpicture}[line cap=round,line join=round,>=triangle 45,x=.8cm,y=.8cm] 		
\draw [line width=1pt] (-10,4)-- (-8,4);
\draw [line width=1pt] (-10,4)-- (-10,2);
\draw [line width=1pt] (-8,4)-- (-8,2);
\draw [line width=1pt] (-10,2)-- (-8,2);
\draw [line width=1pt] (-8,4)-- (-6,3);
\draw [line width=1pt] (-8,2)-- (-6,3);
\draw [line width=1pt] (-10,4)-- (-8,2);
\draw [line width=1pt] (-10,2)-- (-8,4);
\draw [line width=1pt] (-10,4)-- (-6,3);
\draw [line width=1pt] (-10,2)-- (-6,3);
\draw [line width=1pt] (-6,3)-- (-5,4);
\draw [line width=1pt] (-6,3)-- (-4.5,3.7);
\draw [line width=1pt] (-6,3)-- (-4.22,3.3);
\draw [line width=1pt] (-6,3)-- (-4.2,2.76);
\draw [line width=1pt] (-6,3)-- (-4.52,2.32);
\draw [line width=1pt] (-6,3)-- (-5,2);
\begin{scriptsize} 			
\draw [fill=black] (-10,4) circle (2.5pt);
\draw [fill=black] (-8,4) circle (2.5pt);
\draw [fill=black] (-10,2) circle (2.5pt);
\draw [fill=black] (-8,2) circle (2.5pt);
\draw [fill=black] (-6,3) circle (2.5pt);
\draw [fill=black] (-5,4) circle (2.5pt);
\draw [fill=black] (-4.5,3.7) circle (2.5pt);
\draw [fill=black] (-4.22,3.3) circle (2.5pt);
\draw [fill=black] (-4.2,2.76) circle (2.5pt);
\draw [fill=black] (-4.52,2.32) circle (2.5pt);
\draw [fill=black] (-5,2) circle (2.5pt); 			
\end{scriptsize}
\end{tikzpicture}
\caption{Pineapple graph $P_{11,4}=T(0^11^{3}0^{6}1)$}
\label{fig1}
\end{figure*}
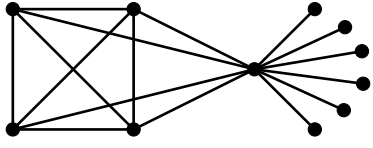
 
A {\em split graph} is a graph whose vertex set $V$ can be bi-partitioned into a clique $V_1$ and an independent set $V_2$. Notice that $V$ may have more than one such partition $(V_1,V_2)$. For convenience, we always assume that the number of vertices in $V_1$ is as small as possible while $V_2$ is an independent set, and we denote such a split graph by $S_{n, n_1}$, where $n=|V|$ and $n_1=\left|V_1\right|$. A {\em complete split graph}, denoted by $C S_{n, n_1}$, is the split graph $S_{n, n_1}$ for which each vertex in $V_2$ is adjacent to each vertex in $V_1$. 


 
The matrix $L(G)=D(G)-A(G)$ is called the {\em Laplacian matrix} of $G$, where $A(G)$ is the adjacency matrix of $G$ and $D(G)=$ $\operatorname{diag}(d_1(G),\ldots, d_n(G))$ is the diagonal matrix of vertex degrees of $G$. It is well known that $L(G)$ is positive semidefinite, and hence, all its eigenvalues are nonnegative real numbers. Let $\mu_1(G) \geq \mu_2(G) \geq \cdots \geq \mu_n(G)$ be the eigenvalues of $L(G)$. Since each row sum of $L(G)$ is $0$, $\mu_n(G)=0$. The eigenvalues of $L(G)$ are known as the {\em Laplacian eigenvalues} of $G$. It is easy to see that $\sum_{i=1}^n \mu_i(G)=2 e(G)$. For $k\in\{1,\ldots, n\}$, let $$S_k(G) = \sum_{i=1}^{k} \mu_i(G)$$
be the sum of the $k$ largest Laplacian eigenvalues of $G$. The parameter $S_k(G)$ sheds light on several fundamental problems considered in spectral graph theory. In particular, the parameter $S_k(G)$ shares a strong connection with the spectral quantity known as the {\em Laplacian energy}, defined by Gutman and Zhou~\cite{GutmanZhou} as
$$\mathrm{LE}(G) = \sum_{i=1}^{n} \left| \mu_i(G) - \frac{2e(G)}{n} \right|.$$

Since its introduction, Laplacian energy has been a widely studied topic. One of the central challenges in this area is to identify, among all $n$-vertex graphs, the one that maximizes the Laplacian energy (see \cite{Vina2013}).

The sum $S_k(G)$ of the $k$ largest Laplacian eigenvalues has been the subject of two important conjectures. Bai\cite{Bai2011} provided a breakthrough result by proving a long-standing conjecture of Grone and Merris \cite{Grone1994}, which can be formally stated as follows. 

\begin{theorem} \label{groneme} {\rm (Grone-Merris-Bai Theorem) } 
Let $G$ be a graph on $n$ vertices and $k$ be any integer such that $1 \leq k \leq n$. Then, 
$$S_k(G)\leq \sum_{i=1}^{k} d_i^*(G).$$
\end{theorem}
The focus of this paper lies in another prominent conjecture connected to the Grone-Merris-Bai framework, namely Brouwer’s conjecture, originally posed by Andries E. Brouwer \cite{Brouwer2012}. The conjecture is formulated as follows.
\begin{Conjecture} {(Brouwer's conjecture)}
{\em Let $G$ be a graph on $n$ vertices and with $e(G)$ edges. Then, for every $k\in\{1, 2, \ldots, n\}$, 
$$S_k(G)=\sum_{i=1}^{k} \mu_i(G)\leq e(G)+\binom{k+1}{2}.$$}
\end{Conjecture} 
It can be observed that for nonsplit graphs, Brouwer’s inequality provides tighter bounds than the bounds provided by Grone-Merris \cite{Grone1994} (see the article \cite{Mayank2010}). As a direct consequence of Theorem 1.1, the following upper bound on the Laplacian energy can be obtained.
\begin{equation}\label{lener}
\mathrm{LE}(G)\leq \sum_{i=1}^{n}\left|d_i^*(G)-\frac{2e(G)}{n}\right|.
\end{equation}

It is well known that the Laplacian eigenvalues of threshold graphs are the conjugates of their degrees (see \cite{Merris1994}). Thus, it follows that the bound (\ref{lener}) is attained by the threshold graphs. This was proved by Helmberg and Trevisan \cite{Helm2015}, and the main tool used in the proof was Theorem \ref{groneme}. It is also reasonable to conceive that the class of threshold graphs is a good class of examples of graphs with the large(st) Laplacian energy. 

The Laplacian energy conjecture \cite{Vina2013} states that among all connected graphs on $n$ vertices, the pineapple graph (a threshold graph), $P_{n,\left\lfloor \frac{2n}{3} \right\rfloor}=T(0^11^{\left\lfloor \frac{2n}{3} \right\rfloor-1}0^{n-\left\lfloor \frac{2n}{3} \right\rfloor-1}1),$ with trace $\left\lfloor \frac{2n}{3}\right\rfloor$ maximizes the Laplacian energy. Among all connected threshold graphs, the pineapple graph is indeed the maximizer (see \cite{Helm2015}). For general threshold graphs on $n$ vertices, the clique of size $\left\lfloor \frac{2n+1}{3} \right\rfloor+1$ together with $\left\lfloor \frac{n-3}{3} \right\rfloor$ isolated vertices is a threshold graph maximizing Laplacian energy \cite{Helm2015}, and it has been conjectured that this graph has maximum Laplacian energy among all graphs on $n$ vertices. 

Helmberg and Trevisan \cite{Helm2017} introduced the concept of spectral threshold dominance in graphs as defined below:
 
\begin{definition}
Let $G$ be a graph on $n$ vertices and with $e(G)$ edges. The graph $G$ is said to be {\em spectrally threshold dominated} if for every $k\in\{1,2,\ldots,n\}$, there exists a threshold graph $T_k$ with the same number of vertices and edges such that
$$\sum_{i=1}^{k}d_i^*(T_k)=\sum_{i=1}^{k}\mu_i(T_k)\geq\sum_{i=1}^{k}\mu_i(G).$$
\end{definition}

In the same article, the authors proved the following result.

\begin{theorem} {\rm(Helmberg and Trevisan \cite{Helm2017})} \label{Hel2017}
For each spectrally threshold dominated graph $G$, there exists a threshold graph with the same number of vertices and edges whose Laplacian energy is at least as large as that of $G$.
\end{theorem}
Furthermore, this notion was shown to be closely related to both Brouwer’s conjecture and the conjecture concerning Laplacian energy. In particular, the following result was demonstrated. 
\begin{theorem} {\rm (Helmberg and Trevisan~\cite{Helm2017})}\label{bro=sthre}
A graph $G$ satisfies Brouwer’s conjecture if and only if it is spectrally threshold dominated.
\end{theorem}

From the above results and discussion, we observe that if Brouwer’s conjecture is true, and hence, every graph is spectrally threshold dominated, then the Laplacian energy of every graph is bounded by the energy of threshold graphs, which implies the Laplacian energy conjecture.

Recently, Torres and  Trevisan \cite{torres} introduced a graph parameter called the \emph{Brouwer critical index}. For a graph $G$ of order $n$, the Brouwer critical index is an integer
$
h \in \{1,\ldots,n\},
$
with the property that if Brouwer’s conjecture holds for $k=h$, then it automatically holds for every
$
1 \leq k \leq n.
$
In other words, verifying the $h$-Brouwer inequality is sufficient to guarantee that the graph $G$ satisfies Brouwer’s conjecture entirely. More precisely, they proved that a graph satisfies Brouwer’s conjecture if and only if it satisfies a particular Brouwer inequality corresponding to this critical index.
 
Motivated by the above results, we give another way to prove Brouwer's conjecture by proving the following result relating to the graph complement.
 
\begin{theorem} \label{samt}
Let $G$ be a graph on $n$ vertices. Then, $G$ satisfies Brouwer’s conjecture for every value of $k$ $(1\leq k \leq n)$ if and only if for a fixed positive integer $h$, $\mathcal{S}_h(\overline{G})\leq e(\overline{G})+\binom{h+1}{2}$ holds whenever $\mathcal{S}_h(G)\leq e(G)+\binom{h+1}{2}$.
\end{theorem}
 \begin{remark}
 	Chen \cite{Chen2019} proved that for a graph $G$ and its complement $\overline{G}$, if the inequality
 	$$
 	S_k(\overline{G}) \leq e(\overline{G}) + \binom{k+1}{2}
 	$$
 	holds for all $k$, then the corresponding inequality
 	$$
 	S_k(G) \leq e(G) + \binom{k+1}{2}
 	$$
 	also holds for all $k$. Chen established a global relation between a graph $G$ and its complement $\overline{G}$ by proving that if Brouwer’s inequality is satisfied for $\overline{G}$ for every $k$, then it is also satisfied for $G$ for every $k$. In contrast, Theorem~\ref{samt} characterizes the validity of Brouwer’s conjecture for all values of $k$ through the verification of the inequality at a fixed index $h$.
 \end{remark}
Numerous partial results have been established in support of Brouwer’s conjecture. For instance, Mayank \cite{Mayank2010} (see also \cite{Chen2018a}) demonstrated that split graphs and cographs satisfy the conjecture. In \cite{Brouwer2012}, Brouwer himself verified its validity for all graphs with up to 10 vertices using computational methods. When $k = 1$, the conjecture follows from the well-known bound $\mu_1(G) \leq n$. It also holds for the extremal cases $k = n$ and $k = n-1$ by simple arguments. Chen \cite{Chen2019} showed that if Brouwer's conjecture holds for all graphs when $k=p$, then Brouwer's conjecture holds for all graphs when $k=n-p-1$ as well, where $p$ is an integer and $1 \leq p \leq \frac{n-1}{2}$. Thus, Brouwer's conjecture also holds for all graphs when $k=n-2$ and $k=n-3$. In addition, it has been proved that Brouwer's conjecture is true for several classes of graphs (for all $k\in\{1,\ldots,n\}$) such as trees \cite{Haemers2010}, threshold graphs \cite{Haemers2010}, unicyclic graphs \cite{Du2012, Fritscher2011}, bicyclic graphs \cite{Du2012}, regular graphs \cite{Mayank2010}, and split graphs \cite{Mayank2010}. Recently, Cooper \cite{Cooper2021} showed that Brouwer's conjecture is true for planar graphs when $k\geq 11$ and for bipartite graphs when $k\geq\sqrt{32n}$. Ferreira \cite{Ferreira2021} showed that if the Brouwer conjecture is not valid for a graph and this graph respects some conditions, the conjecture is also not valid for the graphs obtained by deleting an edge or vertex. Similar results can also be found in \cite{Chen2018b, Ganie2020, Vladimir2019, Wang2012}. 
 
The remainder of the paper is organized as follows. In Section \ref{sec2}, we establish a connection between the matching number of a graph and Brouwer’s conjecture. Specifically, we show that if $G$ is a graph with $n$ vertices, $e(G)$ edges, and matching number $m(G)$, then $\mathcal{S}_k(G) \leq e(G) + \binom{k+1}{2}$ holds for all integers $k$ satisfying $1 \leq k \leq \left\lfloor\frac{m(G)}{2}\right\rfloor$. This improves several known results previously established in the literature. In Section \ref{sec3}, we present a result demonstrating that Brouwer’s conjecture holds for $k=3$ for all graphs except those whose matching number is $3$, $4$, or $5$.  Finally, in Section \ref{sec4}, we introduce a new approach to verifying Brouwer’s conjecture. We prove that a graph $G$ satisfies Brouwer’s conjecture if and only if for a fixed integer $h$, $\mathcal{S}_h(\overline{G})\leq e(\overline{G})+\binom{h+1}{2}$ holds whenever $\mathcal{S}_h(G)\leq e(G)+\binom{h+1}{2}$.

\section{ Brouwer's conjecture and the  matching number }\label{sec2}
We start the section with the following observation.
 
\begin{theorem}\label{weyl}
{\em (Fulton \cite{Fan1949} (see also \cite{Chen2018a}))} Let $A$ and $B$ be two real symmetric $n\times n$ matrices. Then for any $k$, $1 \leq k \leq n$,
$$\sum_{i=1}^k \lambda_i(A+B) \leq \sum_{i=1}^k \lambda_i(A)+\sum_{i=1}^k \lambda_i(B),$$ where $\lambda_i(M)$ denotes the i-th largest eigenvalue of the symmetric matrix $M$. 	
\end{theorem} 

An immediate consequence of Theorem \ref{weyl} is the following corollary which will be used later.

\begin{corollary}\label{weylcor}
Let $G_1, \ldots, G_r$ be some edge disjoint graphs on $n$ vertices. Then, $$\mathcal{S}_k\left(G_1 \cup \cdots \cup G_r\right) \leq \sum_{i=1}^r \mathcal{S}_k\left(G_i\right)$$ for any $k$, $1 \leq k \leq n$.
\end{corollary} 
The following lemma serves as a cornerstone to prove the main result of this section.
 
\begin{lemma}\label{countermat}
If Brouwer's conjecture is false for some $k$,  then there exists a counterexample graph $G$ such that for every subgraph $H$ of $G$ containing at least one edge, we have $S_k(H) > e(H)$.
\end{lemma} 
 
\begin{proof}
Let $G$ be a graph with $n$ vertices having the minimum number of edges and $$
e(G)+\binom{k+1}{2}<S_k(G),$$ for some $k$, $1\leq k\leq n$. 
Suppose that $G$ has a subgraph $H$ with $n_0$ vertices that satisfies $$S_k\left(H\right) \leq e\left(H\right).$$
By  Corollary \ref{weylcor}, we obtain  $$e(G)+\binom{k+1}{2}<S_k(G) \leq S_k\left(H\cup (n-n_0)K_1\right)+S_k\left(G-H \right)=S_k\left(H\right)+S_k\left(G-H \right),$$
where $tK_1$ denotes the disjoint union of 
$t$ isolated vertices. This implies that $$e(G-H)+\binom{k+1}{2}<S_k(G-H),$$ which contradicts the minimality of $e(G)$ and hence, the result follows. 	
\end{proof} 
Rocha and Trevisan \cite{Rocha2014} proved the following result.
\begin{theorem}{\rm (Rocha and Trevisan \cite{Rocha2014})}\label{rochatra}
Let $G$ be a graph with $n$ vertices, $e(G)$ edges, and girth $g$. Then 
$$\mathcal{S}_k(G)\leq e(G)+\binom{k+1}{2}$$ holds for all integers $k$ satisfying $1\leq k\leq \lfloor\frac{g}{5}\rfloor$, where $\lfloor . \rfloor$ denotes the greatest integer function.
\end{theorem}
The above result was later improved by Chen \cite{Chen2018a} as stated below.

\begin{theorem} {\rm (Chen \cite{Chen2018a})}\label{cheng}
Let $G$ be a graph with $n$ vertices, $e(G)$ edges, and girth $g$. Then 
$$\mathcal{S}_k(G)\leq e(G)+\binom{k+1}{2}$$ holds for all integers $k$ satisfying $1\leq k\leq \lfloor \frac{g}{4}\rfloor$.
\end{theorem}
 
It is easy to see that if $G$ is a graph with matching number $m(G)$ and girth $g$, then $m(G) \geq \lfloor g / 2 \rfloor$. Therefore, the following result is an  improvement of both Theorems \ref{rochatra} and \ref{cheng}.

\begin{theorem}\label{bromatc}
Let $G$ be a graph with $n$ vertices, $e(G)$ edges, and matching number $m(G)$. Then $$\mathcal{S}_k(G) \leq e(G)+\binom{k+1}{2}$$ holds for all integers $k$ satisfying  $1 \leq k \leq\lfloor \frac{m(G)} { 2}\rfloor$.
\end{theorem} 
\begin{proof}
Since $G$ is a graph with matching number $m(G)$, $m(G)K_2$ is a subgraph of $G$. Now, one can easily see that for all integers 	$k$ satisfying, $1 \leq k \leq\lfloor \frac{m(G)} {2} \rfloor$,  
$$\mathcal{S}_k(m(G)K_2)=2k\leq e(m(G)K_2).$$ Hence the result follows by Lemma \ref{countermat}.
\end{proof}
An immediate consequence of Theorem \ref{bromatc} is the following result.
\begin{corollary}
If $G$ is a graph on $n$ vertices with a perfect matching, then $$\mathcal{S}_k(G) \leq e(G)+\binom{k+1}{2}$$ holds for all integers $k$ satisfying $1\leq k \leq\lfloor\frac{n}{4}\rfloor$.
\end{corollary} 
 
In \cite{Wang2012}, Wang, Huang, and Liu proved the following result related to the Brouwer's conjecture.
\begin{lemma} {\rm (Wang, Huang, and Liu \cite{Wang2012})}\label{wangk}
If $G$ is a connected graph with $n$ vertices and $e(G)$ edges, then for each integer $k$ satisfying $$\frac{3n-4+\sqrt{8n^2(e(G)-n+1)+(n-4)^2}}{2n}\leq k\leq n,$$ the following inequality holds:
$$\mathcal{S}_k(G)\leq e(G)+\binom{k+1}{2}.$$	
\end{lemma}
Our next result shows that Brouwer's conjecture holds whenever the graph has a large matching compared with the number of cycles. 

\begin{theorem} \label{angtah}
Let $G$ be a connected graph with $n$ vertices and $e(G)$ edges. If
$$m(G)\geq 3+\sqrt{8(e(G)-n+1)+1},$$
then $$\mathcal{S}_k(G)\leq e(G)+\binom{k+1}{2}$$ holds for all integers $k$
satisfying $1\leq k\leq n$. 
\end{theorem}
\begin{proof}
For any integer $k$ such that $1 \leq k \leq \left\lfloor \frac{m(G)}{2} \right\rfloor$, it follows directly from Theorem~\ref{bromatc} that
$$\mathcal{S}_k(G) \leq e(G) + \binom{k+1}{2}.$$
On the other hand, for $\left\lfloor \frac{m(G)}{2} \right\rfloor < k \leq n$, under the condition
$$m(G) \geq 3 + \sqrt{8(e(G) - n + 1) + 1},$$
it can be easily verified that
$$n \geq k \geq \left\lceil \frac{m(G)}{2} \right\rceil \geq \frac{m(G)}{2} > \frac{3n - 4 + \sqrt{8n^2(e(G) - n + 1) + (n - 4)^2}}{2n},$$ where $ \lceil . \rceil$ denotes the least integer function.
Thus, if $G$ is a connected graph with $n$ vertices and $e(G)$, then by using Lemma \ref{wangk}, we have
$$\mathcal{S}_k(G)\leq e(G)+\binom{k+1}{2},$$
for all integers $k$ satisfying $1\leq k\leq n$. This proves the result. 	
\end{proof}

\begin{remark}
Theorem \ref{angtah} implies that Brouwer's conjecture is valid whenever the matching number of the graph is sufficiently large relative to the number of its cycles. It represents an improvement over Theorem \ref{rochatra} by Rocha and Trevisan \cite{Rocha2014}, and Theorem \ref{cheng} by Chen \cite{Chen2018a}. Furthermore, it is known that a connected graph with $n$ vertices and $c$ cycles (a so-called {$c$-cyclic graph}) contains $e(G) = n - 1 + c$ edges. Under this context, we get $$m(G) \geq 3 + \sqrt{8c + 1}.$$
This highlights the fact that for Brouwer's conjecture to hold, the matching number $m(G)$ must increase as the number of cycles $c$ increases. For instance, Brouwer’s conjecture holds for unicyclic graphs with matching number at least $6$; bicyclic and tricyclic graphs with matching number at least $8$; tetracyclic graphs with matching number at least $9$, and so on. 	
\end{remark} 

\section{Brouwer's conjecture for $k=3$}\label{sec3}
Recently, in \cite{Chen2018a}, Chen proved the following results for split graphs.
\begin{lemma} {\em (Chen \cite{Chen2018a})} \label{splitl} For a split graph $S_{n,n_1}$ and for any $1\leq k\leq n$,
	$$\mathcal{S}_k\left(S_{n, n_1}\right)\leq kn_1+e\left(S_{n,n_1}\right)-\frac{n_1\left(n_1-1\right)}{2},$$
	with equality if $1\leq k\leq n_1-1$ and $S_{n, n_1}\cong KS_{n,n_1}$, or if $n_1\leq k\leq n-1$ and $S_{n, n_1}\cong CS_{n,n_1}$.
\end{lemma}
Next, let us state the well-known Kőnig-Egerváry theorem, which is central to the matching theory of bipartite graphs.

\begin{lemma} \label{mat=tau}
If $G$ is a bipartite graph with the matching number $m(G)$ and covering number $\tau(G)$, then $m(G)=\tau(G)$.
\end{lemma}
The next result demonstrates that Brouwer’s conjecture holds for $k=3$ for all graphs except those whose matching number is $3$, $4$, or $5$.

\begin{theorem}\label{rochatr}
	Let $G$ be a graph with $n$ vertices and matching number $m(G)$ such that $m(G)\notin \{3,4,5\}$. Then, $$\mathcal{S}_3(G)\leq e(G)+6.$$	
\end{theorem}
\begin{proof}
	We divide the proof into the following cases.
	\begin{description}
		\item[Case 1.]  If $m(G)=1$, then it is easy to see that either $G=K_{1, n-1}$ or $G=K_3$. In that case, we have  $$\mathcal{S}_3\left(K_{1, n-1}\right)=e\left(K_{1, n-1}\right)+3<e\left(K_{1,n-1}\right)+6,$$ and $$ \mathcal{S}_3\left(K_3\right)=e\left(K_3\right)+3<e\left(K_3\right)+6.$$  Thus, the assertion holds.
		\item[Case 2.]  Let $m(G)=2$. Let $V_\mathcal{T}(G)=\{1,\ldots,{\tau(G)}\}$ be a minimum vertex cover of a graph $G$ such that $\tau(G)=|V_\mathcal{T}(G)|$. Define the spanning subgraphs $G_1,\ldots, G_{\tau(G)}$ of $G$ such that their edge sets are disjoint and collectively form the edge set of $G$ in the following way: 
		$$V(G_i)=V(G),i=1,\ldots,\tau(G),$$ 
		$$E(G_1)=\{\{1,k\}:k\in N_G(1)\}, E(G_i)=\{\{i,k\}: k\in N_G(i)\setminus \{1,\ldots,{i-1}\}\},i=2,\ldots,\tau(G).$$  
		Thus,
		$$E(G)=E(G_1)\cup \cdots \cup E(G_{\tau(G)}).$$  
		Notice that each $G_i$ is a star graph $K_{1,e(G_i)}$ together with $n-e(G_i)-1$ isolated vertices. Thus, 
		$$G_i=K_{1,e(G_i)}\cup (n-e(G_i)-1)K_1.$$  
		Since the Laplacian matrix is additive over edge-disjoint subgraphs, we have
		$$L(G)=L(G_1)+\cdots+L(G_{\tau(G)}).$$  
		From the above construction and the definition of vertex cover, the maximum degrees of the subgraphs satisfy
		$$\sum_{i=1}^{\tau(G)}\Delta(G_i)=e(G),$$ where $\Delta(G)$ denote the maximum vertex degree in a graph $G$.
		
		It is well-known that \cite{Brouwer2012} for a complete bipartite graph ${K_{m,n}}$, the parameters $m$ and $n$ are Laplacian eigenvalues with multiplicities $n-1$ and $m-1$, respectively and the remaining nonzero Laplacian eigenvalue is $m+n$. Thus, we have
		$$
		\mathcal{S}_k(G_i)\leq \Delta(G_i)+k\text{ for each } i= 1, \ldots,\tau(G).
		$$  
		Using Theorem \ref{weyl}, we get
		\begin{equation}\label{tauineq}
			\mathcal{S}_k(G)\leq \sum_{i=1}^{\tau(G)}\mathcal{S}_k(G_i)=\sum_{i=1}^{\tau(G)} (\Delta(G_i) + k)=e(G)+k\tau(G).
		\end{equation} 
		Now, if $G$ is a bipartite graph, then from Equation (\ref{tauineq}) and Lemma \ref{mat=tau}, we have $$\mathcal{S}_3(G)\leq e(G)+6.$$
		
		If $G$ is a non-bipartite graph, then first assume that $G$ contains a subgraph isomorphic to $K_3$. If every edge of $G$ has at least one endpoint in the vertex set $V(K_3)$ of the triangle, then $G$ has the form $S_{n,3}$. Then, by using Lemma \ref{splitl}, the result follows. Now, suppose there exists an edge $e=ab$ whose endpoints lie in $V(G)\setminus V(K_3)$. Let $M=V(G)\setminus \{a,b,u,v,w\}$, where $V(K_3)=\{u, v, w\}$. Since $m(G)=2$, there are no edges between $V(K_3)$ and $M$, and every vertex in $M$ is adjacent to one of the endpoints of $e$, say $a$. Hence, there are no edges between $b$ and the vertices in $M$.  Therefore, $G$ is isomorphic to one of the graphs shown in Figure \ref{fig2}.
		\begin{figure}[!htb]			
			\centering
			\begin{tikzpicture}[line cap=round,line join=round,>=triangle 45,x=.9cm,y=0.9cm]
				
				\draw [line width=1pt] (-12,1.42)-- (-10.04,1.4);
				\draw [line width=1pt] (-10.04,1.4)-- (-11,3);
				\draw [line width=1pt] (-12,1.42)-- (-11,3);
				\draw [line width=1pt] (-11.98,4.48)-- (-11,3);
				\draw [line width=1pt] (-11,3)-- (-9.98,4.52);
				\draw [line width=1pt] (-6,2)-- (-5,3);
				\draw [line width=1pt] (-5,3)-- (-4,2);
				\draw [line width=1pt] (-6,2)-- (-5,1);
				\draw [line width=1pt] (-5,1)-- (-4,2);
				\draw [line width=1pt] (-6,2)-- (-4,2);
				\draw [line width=1pt] (-5,3)-- (-6,4.54);
				\draw [line width=1pt] (-5,3)-- (-4.02,4.54);
				\draw [line width=1pt] (-0.6,2)-- (0.4,3);
				\draw [line width=1pt] (0.4,3)-- (1.4,2);
				\draw [line width=1pt] (-0.6,2)-- (0.4,1);
				\draw [line width=1pt] (0.4,1)-- (1.4,2);
				\draw [line width=1pt] (-0.6,2)-- (1.4,2);
				\draw [line width=1pt] (0.4,3)-- (-0.6,4.54);
				\draw [line width=1pt] (0.4,3)-- (1.38,4.54);
				\draw [line width=1pt] (0.4,3)-- (0.4,1);
				
				\draw (-11.3,1) node[anchor=north west] {$G_1$ };
				\draw (-5.3,.8) node[anchor=north west] {$G_2$ };
				\draw (0,.8) node[anchor=north west] {$G_3$ };
				
				\begin{scriptsize}
					\draw [fill=black] (-11,3) circle (2.5pt);
					\draw [fill=black] (-12,1.42) circle (2.5pt);
					\draw [fill=black] (-10.04,1.4) circle (2.5pt);
					\draw [fill=black] (-11.98,4.48) circle (2.5pt);
					\draw [fill=black] (-9.98,4.52) circle (2.5pt);
					\draw [fill=black] (-11.5,4.56) circle (1.5pt);
					\draw [fill=black] (-11,4.56) circle (1.5pt);
					\draw [fill=black] (-10.56,4.56) circle (1.5pt);
					\draw [fill=black] (-5,3) circle (2.5pt);
					\draw [fill=black] (-4,2) circle (2.5pt);
					\draw [fill=black] (-6,2) circle (2.5pt);
					\draw [fill=black] (-5,1) circle (2.5pt);
					\draw [fill=black] (-6,4.54) circle (2.5pt);
					\draw [fill=black] (-4.02,4.54) circle (2.5pt);
					\draw [fill=black] (-5.58,4.58) circle (1.5pt);
					\draw [fill=black] (-5,4.6) circle (1.5pt);
					\draw [fill=black] (-4.48,4.56) circle (1.5pt);
					\draw [fill=black] (0.4,3) circle (2.5pt);
					\draw [fill=black] (1.4,2) circle (2.5pt);
					\draw [fill=black] (-0.6,2) circle (2.5pt);
					\draw [fill=black] (0.4,1) circle (2.5pt);
					\draw [fill=black] (-0.6,4.54) circle (2.5pt);
					\draw [fill=black] (1.38,4.54) circle (2.5pt);
					\draw [fill=black] (-0.18,4.58) circle (1.5pt);
					\draw [fill=black] (0.4,4.6) circle (1.5pt);
					\draw [fill=black] (0.92,4.56) circle (1.5pt);
				\end{scriptsize}
			\end{tikzpicture}
			\caption{\label{fig2}}				
		\end{figure}  
		If $G\cong G_1$, then $G$ is a unicyclic graph, and we have $\mathcal{S}_3(G) \leq e(G)+6$. If $G\cong G_i$ for $i=2,3$, then $G$ has the form $S_{n,3}$ and $S_{n,4}$, respectively. Thus, by Lemma \ref{splitl}, we have $\mathcal{S}_3(G)\leq e(G)+6$, completing the proof in this case.		
		
		\item[Case 3.]  If $m(G)\geq 6$, then the result follows directly from Theorem \ref{bromatc}.				
	\end{description}
	This completes the proof.
\end{proof}
A connected graph on $n$ vertices is called \emph{tricyclic} if it has exactly $n+2$ edges. The following result was established by Wang, Huang, and Liu \cite{Wang2012}.

\begin{lemma}\label{Waang}
	{\em (Wang, Huang and Liu)}
	Let $G$ be a tricyclic graph on $n$ vertices. Then,
	$$
	\mathcal{S}_k(G)\leq e(G)+\binom{k+1}{2}
	$$
	holds for all integers $k\neq 3$ with $1\leq k\leq n$.
\end{lemma}

Combining Theorem~\ref{rochatr} with Lemma~\ref{Waang}, we immediately obtain the following corollary.

\begin{corollary}
	Let $G$ be a tricyclic graph on $n$ vertices with matching number $m(G)\notin \{3,4,5\}$. Then,
	$$
	\mathcal{S}_k(G)\leq e(G)+\binom{k+1}{2}
	$$
	holds for all integers $k$ satisfying $1\leq k\leq n$.
\end{corollary}
The following result relates the Laplacian eigenvalues of a graph and those of its complement. 
\begin{lemma} \label{lapcom}{\em (Merris \cite{Merris1998})} 
	Let $G$ be a graph on $n$ vertices and $\overline{G}$ be its complement. Then $$\mu_n(\overline{G})=0 \mbox{ and } \mu_i(\overline{G})=n-\mu_{n-i}(G) \mbox{ for }i=1,2,\ldots,n-1.$$
\end{lemma} 
We end this section with the following result which establishes that if the Brouwer’s conjecture holds for all graphs when $k=n-p-1$, then it also holds for all graphs when $k=p$. Note that it is the converse of Theorem $3.1$ provided in \cite{Chen2019}.
\begin{theorem} \label{revchen0} 
	Let $p$ be a positive integer such that $1\leq p\leq n-2$. If the Brouwer's conjecture holds for any graph $G$ when $k=n-p-1$, that is, for any graph $G$ with $n$ vertices and $e(G)$ edges,
	\begin{equation}\label{n-p-1}
		\mathcal{S}_{n-p-1}(G)\leq e(G)+\binom{n-p}{2}, 
	\end{equation}
	then the Brouwer's conjecture also holds for $k=p$, that is, for the  graph $G$,
	\begin{equation}\label{pa}
		\mathcal{S}_{p}(G)\leq e(G)+\binom{p+ 1}{2}.
	\end{equation}
	Moreover, if the equality holds in (\ref{n-p-1}) if and only if $G\cong H$$($for some graph $H)$, then the equality holds in (\ref{pa}) if and only if $G\cong \overline{H}$.	
\end{theorem}
\begin{proof}
	Since $\sum_{i=1}^{n-1}\mu_i(G)=2e(G)$, we have $\mathcal{S}_{p}(G)=2e(G)-\sum_{i=p+1}^{n-1}\mu_i(G).$
	Using Lemma \ref{lapcom}, we have 
	\begin{align*}
		\mathcal{S}_{p}(G)
		&=2e(G)-\sum_{i=p+1}^{n-1}\left(n-\mu_{n-i}(\overline{G})\right) \quad \\
		&=2e(G)-\sum_{i=p+1}^{n-1}n+\sum_{i=p+1}^{n-1} \mu_{n-i}(\overline{G})\\
		&=2e(G)-(n-p-1)n+\mathcal{S}_{n-p-1}(\overline{G}).
	\end{align*}
	By Condition (\ref{n-p-1}), we obtain
	\begin{align*}
		\mathcal{S}_{p}(G)
		&\leq 2e(G)-(n-p-1)n+e(\overline{G})+\binom{n-p}{2}\\
		&=e(G)+e(G)+e(\overline{G})-(n-p-1)n+\binom{n-p}{2}\\
		&=e(G)+\binom{n}{2}-(n-p-1) n+\binom{n-p}{2}\\
		&=e(G)+\binom{p+1}{2},
	\end{align*}
	which shows that if the Brouwer's conjecture holds for all graphs when  $k=n-p-1$, then it also holds for all graphs when $k=p$. Now, suppose  that the equality holds in (\ref{n-p-1}) if and only if $G\cong H$. Then, $\mathcal{S}_{n-p-1}(\overline{G})= e(\overline{G})+\binom{n-p}{2}$ holds if and only if $\overline{G}\cong H$, that is, $G\cong \overline{H}$. This yields that the equality holds in (\ref{pa}) if and only if $G \cong \overline{H}$.
\end{proof}
\section{Proof of the Theorem \ref{samt}}\label{sec4}

To prove Theorem \ref{samt}, it is sufficient to prove that if for a fixed positive integer $h$ and a graph $G$, $\mathcal{S}_h(\overline{G})\leq e(\overline{G})+\binom{h+1}{2}$ holds whenever $\mathcal{S}_h(G)\leq e(G)+\binom{h+1}{2}$, then the graph $G$ satisfies Brouwer’s conjecture. Thus, under the given condition, we prove that Brouwer's conjecture is true. For that, we need the following sequence of results under this given condition. Note that all the results considered in this section are assumed under this given condition.

\begin{lemma} \label{revchen} 
Let $p$ be a positive integer with $1\leq p\leq n-2$. If the Brouwer's conjecture holds for a graph $G$ when $k=n-p-1$, that is, for the graph $G$ with $n$ vertices and $e(G)$ edges,
\begin{equation*}\label{an-p-1}
\mathcal{S}_{n-p-1}(G)\leq e(G)+\binom{n-p}{2}, 
\end{equation*}
then the Brouwer's conjecture also holds for $k=p$, that is, for the same graph $G$,
\begin{equation*}\label{p} 	
\mathcal{S}_{p}(G) \leq e(G)+\binom{p + 1}{2}.
\end{equation*}
\end{lemma}
\begin{proof}
Since $\sum_{i=1}^{n-1}\mu_i(G)=2 e(G)$, it follows that
$$\mathcal{S}_{p}(G)=2 e(G)-\sum_{i=p+1}^{n-1} \mu_i(G).$$
Using Lemma \ref{lapcom}, we have 
\begin{align*}
\mathcal{S}_{p}(G)
& =2 e(G)-\sum_{i=p+1}^{n-1}\left(n-\mu_{n-i}(\overline{G})\right) \quad \\
&=2e(G)-\sum_{i=p+1}^{n-1}n+\sum_{i=p+1}^{n-1}\mu_{n-i}(\overline{G})\\
&= 2e(G)-(n-p-1)n+\mathcal{S}_{n-p-1}(\overline{G}).
\end{align*}
Under the assumption that $\mathcal{S}_h(\overline{G})\leq e(\overline{G})+\binom{h+1}{2}$ holds whenever $\mathcal{S}_h(G)\leq e(G)+\binom{h+1}{2}$, we obtain
\begin{align*}
\mathcal{S}_{p}(G)
& \leq 2 e(G)-(n-p-1) n+e(\overline{G})+\binom{n-p}{2} \\
&= e(G)+e(G)+e(\overline{G})-(n-p-1) n+\binom{n-p}{2}  \\
& =e(G)+\binom{n}{2}-(n-p-1) n+\binom{n-p}{2} \\
& =e(G)+\binom{p+1}{2}.
\end{align*}
This shows that if the Brouwer's conjecture holds for a graph when $k=n-p-1$, then it also holds for the same graph when $k=p$. 
\end{proof}
The next result shows that for a graph with a large number of vertices compared to its covering number, Brouwer's conjecture holds for various values of $k$. This provides partial confirmation of Brouwer's conjecture under specific structural conditions.
\begin{lemma}\label{brotau} Let $G$ be a graph with $n$ vertices and $e(G)$ edges having covering number $\tau(G)$. For a positive integer $p$, if $n\geq 2\tau(G)+p$, then $$\mathcal{S}_k(G)\leq e(G)+\binom{k+1}{2}$$ holds for all integers $k$ satisfying $1\leq k\leq p$.
\end{lemma}
\begin{proof}  

Using Equation (\ref{tauineq}), we get
\begin{equation*}\label{xtauineq}
\mathcal{S}_k(G)\leq \sum_{i=1}^{\tau(G)}\mathcal{S}_k(G_i)=\sum_{i=1}^{\tau(G)} (\Delta(G_i) + k)=e(G)+k\tau(G).
\end{equation*} 
Now, for $k=n-\alpha-1 $ $(1\leq \alpha \leq p)$, assume that
\begin{equation}\label{maineq}
\mathcal{S}_k(G)\leq e(G)+k\tau(G) \leq e(G)+\binom{k+1}{2}.
\end{equation} 
To ensure the last inequality holds in \ref{maineq}, we require
$$(n-\alpha-1)\tau(G)\leq \binom{n-\alpha}{2},$$  
which simplifies to
$$n\geq 2\tau(G)+\alpha.$$  
Thus, under the condition $n\geq 2\tau(G)+p $, we have  $n\geq 2\tau(G)+p\geq 2\tau(G)+\alpha$,  therefore the result follows by Lemma \ref{revchen}. This completes the proof. 	  
\end{proof}

The relationship between the matching number and covering number of a graph is established in the following result. Although a proof already exists in the literature, we provide a simpler one here.
\begin{lemma}\label{mat and tau}
Let $G$ be a graph with the matching number $m(G)$ and covering number $\tau(G)$. Then $$\tau(G)\leq 2m(G).$$
\end{lemma} 
\begin{proof}
Let $M$ be a maximum matching in $G$ with $|M|=m(G)$. Let $V(\mathcal{M}) $ be the set of vertices corresponding to the edges in the maximum matching. That is, for each edge $\{i,j\} \in M$, both $i$ and $j$ belong to $V(\mathcal{M})$. Then, $|V(\mathcal{M})|=2m(G)$. Clearly, every edge in $M$ is  covered by $V(\mathcal{M})$. Since $M$ is maximum, any edge not in $M$ must share at least one vertex with an edge in $M$ (otherwise $M$ wouldn't be maximal). So, every edge in $G$ is incident to at least one vertex in $V(\mathcal{M})$. Hence, $V(\mathcal{M})$ is a vertex cover. So, the minimum vertex cover $\tau(G)$ must satisfy $\tau(G)\leq |V(\mathcal{M})|=2m(G)$ and this proves the result.
\end{proof}
Our next result verifies Brouwer’s conjecture for all $k$ up to a given positive integer $p$ under a linear constraint on the number of vertices in terms of $p$.  
\begin{lemma}\label{my1}
Let $G$ be a graph with $n$ vertices and $e(G)$ edges. If $p$ is a positive integer such that $n\geq 9p-4$, then $$\mathcal{S}_k(G)\leq e(G)+\binom{k+1}{2}$$ holds for all integers $k$ satisfying $1\leq k\leq p$.
\end{lemma} 
\begin{proof}
Let $1\leq k\leq p$. By Theorem \ref{bromatc}, Brouwer’s conjecture holds for all graphs whose matching number is at least $2k$. Thus, it suffices to consider graphs with matching number at most $2k-1$. If $n\geq 9p-4$, then by Lemma \ref{mat and tau}, we have  $$n\geq 9p-4\geq 8k-4+p\geq 4m(G)+p\geq 2\tau(G)+p.$$  
Hence, by using Lemma \ref{brotau}, the result follows, establishing that Brouwer’s conjecture holds for $k$. This completes the proof.
\end{proof}
\begin{corollary}\label{prince}
If $G$ is a graph with $n$ vertices and $e(G)$ edges, then $$\mathcal{S}_k(G)\leq e(G)+\binom{k+1}{2}$$ holds for all integers $k$ satisfying $1\leq k\leq\lfloor \frac{ n+4} {9}\rfloor$.
\end{corollary} 
\begin{proof}
If $p=\lfloor \frac{n+4} {9}\rfloor$, then it is clear that $n\geq 9p-4$.  Thus, the result follows by using Lemma \ref{my1}.
\end{proof}
Now, we prove that if for a fixed integer $k$ and a graph $G$, $\mathcal{S}_k(\overline{G})\leq e(\overline{G})+\binom{k+1}{2}$ holds whenever $\mathcal{S}_k(G)\leq e(G)+\binom{k+1}{2}$, then the graph $G$ satisfies Brouwer’s conjecture.
\begin{proof}[\bf Proof of the Theorem \ref{samt}]
Let $G$ be a graph with $n$ vertices. Now, construct a new graph $H=G\cup(8n-4)K_1$. It is clear that $E(H)=E(G)$ and $|V(H)|=9n-4$. Also,  for $1\leq k \leq n$, it is easy to see that $$S_k(G)=S_k(H).$$ Using Corollary \ref{prince} on the graph $H$, we obtain $$\mathcal{S}_k(G)=\mathcal{S}_k(H)\leq e(H)+\binom{k+1}{2}=e(G)+\binom{k+1}{2}$$ holds for all integers $k$ satisfying $1\leq k\leq n$. This completes the proof.
\end{proof}
\subsection{Alternative proof of Theorem \ref{samt} for bipartite graphs}
We provide an additional proof of Theorem \ref{samt} related to Brouwer's conjecture for bipartite graphs. First, we establish Brouwer’s conjecture for all $k\leq p$ in the case of bipartite graphs, subject to a linear bound on the number of vertices in terms of $p$. Note that this result is a further refinement of Lemma \ref{my1} when $G$ is bipartite.

\begin{lemma} \label{my2}
Let $G$ be a bipartite graph with $n$ vertices and $e(G)$ edges. If $p$ is a positive integer such that $n\geq 5p-2$, then $$\mathcal{S}_k(G)\leq e(G)+\binom{k+1}{2}$$ holds for all integers $k$ satisfying $1\leq k\leq p$.
\end{lemma}

\begin{proof}
Let $1\leq k\leq p$. By Theorem \ref{bromatc}, Brouwer’s conjecture holds for all graphs with matching number at least $2k$. Thus, it suffices to consider graphs with matching number at most $2k-1$. If $n\geq 5p-2$, then by Lemma \ref{mat=tau}, we have  
$$n\geq 5p-2\geq 4k-2+p\geq 2m(G)+p=2\tau(G)+p.$$  
Hence, by using Lemma \ref{brotau}, the result follows.
\end{proof}
The following is an immediate corollary.
\begin{corollary}\label{prince1}
If $G$ is a bipartite graph with $n$ vertices and $e(G)$ edges, then $$\mathcal{S}_k(G)\leq e(G)+\binom{k+1}{2}$$ holds for all integers $k$ satisfying $1\leq k\leq\lfloor \frac{n+2}{5}\rfloor$.
\end{corollary} 
\begin{proof}
If $p=\lfloor \frac{n+2}{5}\rfloor$, then it is clear that $n\geq 5p-2$. Thus, the result follows by applying Lemma \ref{my2}.
\end{proof}
Now, using the above results, we provide a proof for Theorem \ref{samt} for bipartite graphs.
\begin{proof}[2nd proof of Theorem \ref{samt} for bipartite graphs]
Let $G$ be a bipartite graph on $n$ vertices. Now, construct a new bipartite graph $H=G\cup(4n-2)K_1$. It is clear that $|V(H)|=5n-2$ and $E(H)=E(G)$. Also,  for $1\leq k\leq n$, it is easy to see that $S_k(G)=S_k(H).$ Thus, applying Corollary \ref{prince1} on the graph $H$, we have $$\mathcal{S}_k(G)=\mathcal{S}_k(H) \leq e(H)+\binom{k+1}{2}=e(G)+\binom{k+1}{2}$$ holds for all integers $k$ satisfying $1\leq k\leq n$. This completes the proof.
\end{proof}

\bigskip

\bigskip
\noindent {\bf Data availability.} No data was used for the research described in the article.

\bigskip
\noindent {\bf Disclosure statement.} The authors do not have any conflicts of interest.

\bigskip
	
\noindent{\bf Acknowledgements.}  The corresponding author is thankful to IIT Bhubaneswar, India for providing the post doctoral fellowship (Grant No.: F.15-12/2021-Acad/SBS-PDF-01).

\end{document}